\documentclass[12pt]{amsart} %amsfonts,
\usepackage{amsmath,amssymb,amscd,amsthm}
\usepackage{latexsym}
\usepackage{graphicx}
\usepackage[english]{babel}
\usepackage[latin1]{inputenc}       % allow Latin1 characters
\usepackage{textcomp}
\usepackage{times}
\usepackage{pgf,pgfarrows,pgfnodes,pgfautomata,pgfheaps}

\usepackage[a4paper,twoside,left=3cm,right=2.8cm,top=3.1cm,bottom=2.3cm]{geometry}

\newtheorem{theorem}{Theorem}[section]
\newtheorem{corollary}[theorem]{Corollary}
\newtheorem{proposition}[theorem]{Proposition}
\newtheorem{lemma}[theorem]{Lemma}

\newtheorem{definition}[theorem]{Definition}

\def\irr#1{{\rm Irr}(#1)}
\def\irrr#1#2 {\irr {#1 \mid #2}}
\newcommand{\R}{\mathbb R}

\newcommand{\sfe}{{{\mathbb S}^{n-1}}}

\newcommand{\E}{\mathbb E}
\newcommand{\B}{\mathcal{B}}

\def \RR {\mathbb R}
\def \ZZ {\mathbb Z}
\def \cF {\mathcal F}

\begin{document}

\title[Koldobsky's slicing inequality]{The lower bound for Koldobsky's slicing inequality via random rounding}
\author[Bo'az Klartag, Galyna V. Livshyts]{Bo'az Klartag, Galyna V. Livshyts}

%\address{School of Mathematics, Georgia Institute of Technology} \email{glivshyts6@math.gatech.edu}

\subjclass[2010]{Primary: 52}
\date{\today}
\begin{abstract}
We study the lower bound for Koldobsky's slicing inequality. We show that there exists a measure $\mu$ and a symmetric convex body $K \subseteq \RR^n$, such that for all $\xi\in\sfe$ and all $t\in\R,$
$$\mu^+(K\cap(\xi^{\perp}+t\xi))\leq \frac{c}{\sqrt{n}}\mu(K)|K|^{-\frac{1}{n}}.$$
Our bound is optimal, up to the value of the universal constant. It  improves slightly upon the results of the first named author and Koldobsky \cite{KlKol}, which included a doubly-logarithmic error.
The proof is based on an efficient way of discretizing the unit sphere.
%Some further applications of it include estimates on the singular values of random matrices; in the last section, we outline some of the ideas in this regard, and a detailed discussion of the matter shall appear in our separate manuscript \cite{KLsingval}.
\end{abstract}
\maketitle
%\tableofcontents

\section{Introduction}

We shall work in the Euclidean $n$-dimensional space $\R^n$. The unit ball shall be denoted by $B_2^n$ and the unit sphere by $\sfe$. The Lebesgue volume of a measurable set $A\subset \R^n$ is denoted by $|A|$. Throughout the paper, $c$, $C$, $C'$ etc stand for positive absolute constants whose value may change from line to line.

\medskip Given a measure $\mu$ with a continuous density $f$ on $\R^n$ and a
set $A \subseteq \RR^n$ of Hausdorff dimension $n-1$, we write
$$\mu^+(A)=\int_A f(x)dx,$$
where the integration is with respect to the $(n-1)$-dimensional Hausdorff measure.

\medskip
For a measure $\mu$ on $\R^n$ with a continuous density and for an origin symmetric convex body $K$ in $\R^n$ (i.e., $K = -K$), define the quantity
$$S_{\mu,K}=\inf_{\xi\in\sfe} \frac{\mu(K)}{|K|^{\frac{1}{n}}\mu^+(K\cap \xi^{\perp})},$$
where $\xi^{\perp} = \{ x \in \RR^n \, , \, \langle x, \xi  \rangle = 0 \}$ is the hyperplane orthogonal to $\xi$. We let
$$S_n=\sup_{\mu} \sup_{K\subset \R^n} S_{\mu, K}, $$
where the suprema run over measures $\mu$ with a continuous density $f$ in $\RR^n$ and all origin-symmetric convex bodies $K \subseteq \RR^n$

\medskip
Koldobsky in a series of papers \cite{Kol2}, \cite{Kol3}, \cite{Kol4} investigated the question \emph{how large can $S_n$ be?} The discrete version of this question was studied by Alexander, Henk, Zvavitch \cite{AHZ} and Regev \cite{regev}. In \cite{Kol2}, where the question has first arisen, Koldobsky gave upper and lower bounds on $S(\mu,K),$ that are independent of the dimension in the case when $K$ is an intersection body. In \cite{Kol3}, he established the general bound $S_n\leq \sqrt{n}$. In \cite{Kol4}, he has shown that $S_{\mu, K}$ is bounded from above by an absolute constant in the case when $K$ is an unconditional convex body (invariant under coordinate reflections). Further, Koldobsky and Pajor \cite{KolPaj} have shown that $S_{\mu,K}\leq C\sqrt{p}$ when $K$ is the unit ball of an $n$-dimensional section of $L_p.$

\medskip In the case when $\mu$ is the Lebesgue measure, it was conjectured by Bourgain \cite{Bou1}, \cite{Bou2} that $S_{\mu, K}\leq C,$ for an arbitrary origin-symmetric convex body $K$. The best currently known bound in this case is $S_{\mu, K}\leq Cn^{\frac{1}{4}},$ established by the first named author \cite{Kl_slicing}, slightly improving upon Bourgain's estimate from \cite{Bou3}. However, it was shown by the first named author and Koldobsky \cite{KlKol} that $S_n\geq \frac{c\sqrt{n}}{\sqrt{\log\log n}}$. Moreover, it was shown there that for every $n$ there exists a measure $\mu$ with continuous density and a symmetric convex body $K \subseteq \RR^n$ such that for all $\xi\in\sfe$ and for all $t\geq 0,$
\begin{equation}\label{KK1}
\mu^+(K\cap(\xi^{\perp}+t\xi))\leq C\frac{\sqrt{\log\log n}}{\sqrt{n}}\mu(K)|K|^{-\frac{1}{n}},
\end{equation}
where $C>0$ is some absolute constant.
Here $A + x = \{ y + x \, ; \, y \in A \}$ for a set $A \subseteq \RR^n$ and a vector $x \in \RR^m$. In this note we improve the bound (\ref{KK1}), and obtain:

\begin{theorem}\label{main}
For every $n$ there exists a measure $\mu$ and a convex symmetric body $L \subseteq \RR^n$ such that for all $\xi\in\sfe$ and for all $t\geq 0,$
\begin{equation}\label{KK}
\mu^+(L\cap(\xi^{\perp}+t\xi))\leq \frac{C}{\sqrt{n}}\mu(L)|L|^{-\frac{1}{n}},
\end{equation}
where $C>0$ is a universal constant.
\end{theorem}

In \cite{BKK}, the first named author, Bobkov and Koldobsky explored the connections of (\ref{KK1}) and the maximal ``distance'' of convex bodies to subspaces of $L_p$.
Write $L_p^n$ for the collection
of origin-symmetric convex bodies in $\RR^n$ that are linear images of unit balls of $n$-dimensional subspaces of the Banach space $L_p$.
The \emph{outer volume ratio} of a symmetric convex body $K$ in $\R^n$ to the subspaces of $L_p$ is defined as
$$d_{ovr}(K, L_p^n):=\inf_{D \in L_p^n:\,K\subset D} \left(\frac{|D|}{|K|}\right)^{\frac{1}{n}}.$$
John's theorem, and the fact that $l_2^n$ embeds in $L_p$, entails that $d_{ovr}(K, L_p^n)\leq \sqrt{n}$, for any symmetric convex body $K$. Combined with the consideration from \cite{BKK}, Theorem \ref{main} implies a doubly-logarithmic improvement of a result of \cite{BKK}:

\begin{corollary}
There exists an absolute constant $c>0$ and an origin-symmetric convex body $L$ in $\R^n$ such that for any $p \geq 1$,
$$d_{ovr}(L, L_p^n)\geq c\frac{\sqrt{n}}{\sqrt{p}}.$$
\end{corollary}

The construction of $\mu$ and $K$ is randomized, and follows the idea from \cite{KlKol}. The question boils down to estimating the supremum of a certain random function. The method of the proof is based on an efficient way of discretizing the unit sphere. We consider, for every point in $\sfe$, a ``rounding'' to a point in a scaled integer lattice, chosen at random, see Raghavan and Thompson \cite{RagTho}. This construction was recently used in \cite{KA} for efficiently computing sketches of high-dimensional data. It is somewhat reminiscent of the method used in discrepancy theory, called jittered sampling. For instance, using this method, Beck \cite{beck} has obtained strong bounds for the $L_2$-discrepancy.

\medskip
In Section 2 we  describe the net construction. In Section 3 we derive the key estimate for our random function. In Section 4 we conclude the proof of Theorem \ref{main}. In Section 5 we briefly outline some further applications, in particular in relation to random matrices; this discussion in detail shall appear in a separate paper.

\medskip
We use the notation $\log^{(k)} (\cdot)$ for the logarithm iterated $k$ times, and $\log^* n$ for the smallest positive integer $m$ such that $\log^{(m)} n\leq1$.
Denote $\| x \|_p = \left( \sum_i |x_i|^p \right)^{1/p}$ for $x \in \RR^n$, and also $\| x \|_{\infty} = \max_i |x_i|$ and  $|x| = \| x \|_2 = \sqrt{\langle x, x \rangle}$.
Write $B_p^n = \{ x \in \RR^n \, ; \, \| x \|_p \leq 1 \}$. We also write $A + B = \{ x + y \, ; \, x \in A, y \in B \}$ for the Minkowski sum.

\medskip
\textbf{Acknowledgements.}  The second named author is supported in part by the NSF CAREER DMS-1753260. The work was partially supported by the National Science Foundation under Grant No. DMS-1440140 while the authors were in residence at the Mathematical Sciences Research Institute in Berkeley, California, during the Fall 2017 semester. The authors are grateful to Alexander Koldobsky for fruitful discussions and helpful comments. The authors are thankful to the anonymous referee for valuable suggestions.

\section{The random rounding and the net construction}

We fix a dimension $n$ and a parameter  $\rho \in (0,1/2)$.
We define $\cF_{\rho}$
as
the set of all vectors of Euclidean norm between $1-2\rho$ and $1+\rho$ in which every
coordinate is an integer multiple of $\rho/\sqrt{n}$. That is,
%$$ \cF_{\rho} = 2 B_2^n \cap \frac{\rho}{\sqrt{n}} \ZZ^n. $$
$$ \cF_{\rho} = \left( (1+\rho)B_2^n\setminus(1-2\rho)B_2^n \right)\cap \frac{\rho}{\sqrt{n}} \ZZ^n. $$

\begin{lemma}\label{net}
The set $\cF_{\rho}$ satisfies  $\#\mathcal{F}_{\rho}\leq \left(\frac{C}{\rho}\right)^n$, where $C$ is a universal constant.
	Moreover, let $\xi \in \sfe$,
	and suppose that $\eta \in (\rho / \sqrt{n}) \ZZ^n$ satisfies $\| \xi - \eta \|_{\infty} \leq \rho/\sqrt{n}$. Then $\eta \in \cF_{\rho}$.
\end{lemma}

\begin{proof} Any $x \in \cF_{\rho}$ satisfies $\| x \|_1 \leq \sqrt{n} |x| \leq 2 \sqrt{n}$. Hence all vectors in the scaled set $(\sqrt{n} / \rho) \cdot \cF_{\rho}$ have integer coordinates whose absolute values sum to a number which is at most $2 n / \rho$. Recall that the number of vectors $x \in \RR^n$ with non-negative, integer coordinates and $\| x \|_1 \leq R$ equals
	$$ \left( \! \! \begin{array}{c} R + n \\ n \end{array} \! \! \right) \leq \left(e \frac{R+n}{n} \right)^n $$
	where $R$ is a non-negative integer. Consequently,
	$$ \# \cF_\rho \leq 2^n \cdot \left(e \frac{2 \rho^{-1} n +n}{n} \right)^n \leq \left( \frac{C}{\rho} \right)^n. $$
	We move on to the ``Moreover'' part. %Since $\rho < 1/2$
	We have $|\xi - \eta| \leq \sqrt{n} \| \xi - \eta \|_{\infty} \leq \rho$. Therefore $1-2\rho < |\eta| \leq \rho$
	and consequently $\eta \in  ((1+\rho)B_2^n \setminus (1-2\rho)B_2^n) \cap \frac{\rho}{\sqrt{n}} \ZZ^n = \cF_\rho$.%$\eta\in \cF_\rho$.
\end{proof}

\begin{definition}
	For $\xi \in \sfe$
	 consider a random vector $\eta^{\xi} \in (\rho / \sqrt{n}) \ZZ^n$ with independent coordinates such that $\| \xi - \eta^{\xi} \|_{\infty} \leq \rho/\sqrt{n}$
	 with probability one and $\E \eta^{\xi}=\xi$. Namely,
	 for $i=1,\ldots,n$, writing $\xi_i =
	 \frac{\rho}{\sqrt{n}} (k_i + p_i)$ for an integer $k_i$ and $p_i \in [0,1)$,
	\[
	\eta^{\xi}_i=
	\begin{cases}
	\frac{\rho}{\sqrt{n}} k_i ,& \text{with probability}\,\,\, 1-p_i \\
	\frac{\rho}{\sqrt{n}} (k_i+1), & \text{with probability}\,\,\, p_i.
	\end{cases}
	\]
\end{definition}

For any $\xi \in \sfe$, the random vector $\eta^{\xi}$ belongs to $\cF_\rho$ with probability one, according to Lemma \ref{net}. The random vector $\eta^{\xi} - \xi$ is a centered random vector with independent coordinates, all belonging to the interval $[-\rho/\sqrt{n}, \rho/\sqrt{n}]$.
We shall make use of Hoeffding's inequality for bounded random variables (see, e.g., Theorem 2.2.6 and Theorem 2.6.2 in Vershynin \cite{Versh}).

\begin{lemma}[Hoeffding's inequality]\label{hoeffding}
Let $X_1,...,X_n$ be independent random variables taking values in $[m_i,M_i]$, $i=1,...,n$. Then for any $\beta>0,$
$$P\left( \left|\sum_{i=1}^n X_i - \mathbb{E} X_i \right| \geq \beta\right)\leq 2e^{-\frac{c\beta^2}{\sum_{i=1}^n (M_i-m_i)^2}},$$
where $c>0$ is an absolute constant.
\end{lemma}

The next Lemma follows immediately from Hoeffding's inequality with $X_i=(\eta^{\xi}_i-\xi_i)\theta_i$ and $[m_i, M_i]=[-\frac{\rho}{\sqrt{n}}\theta_i, \frac{\rho}{\sqrt{n}}\theta_i]$:

\begin{lemma}\label{subgauss}
	For any $\xi \in \sfe, \beta>0$ and $\theta\in\R^n,$

$$P(|\langle \eta^{\xi}-\xi, \theta\rangle| \geq \beta)\leq 2 \exp\left(-\frac{cn\beta^2}{|\theta|^2\rho^2} \right).$$
Here $c>0$ is an absolute constant.
\end{lemma}

\section{The key estimate}

Let $N$ be  a positive integer, and consider independent random vectors $\theta_1,...,\theta_N$ uniformly distributed on the unit sphere $\sfe$. Unless specified otherwise, the expectation and the probability shall be considered with respect to their distribution.

For $r>0,$ abbreviate $$ \varphi(r)=e^{-\frac{r^2}{2}}.$$% \qquad \qquad (r > 0). $$
The main result of this section is the following Proposition.

\begin{proposition}\label{key}
There exist absolute constants $C_1,\ldots,C_5> 0$ with $C_2 \geq 2 C_5$ and the following property. Let $n \geq 5$, consider $r\in[C_2\sqrt{n}, n]$ and suppose that
$N \geq n$ satisfies $N\in[C_1n\log\frac{Nr}{n\sqrt{n}}, n^{10}]$. Then with probability at least $1-e^{-5n}$, for all $\xi\in\sfe$, and for all $t\in\R,$
$$\frac{1}{N}\sum_{k=1}^N \varphi(r\langle \xi, \theta_k\rangle+t)\leq C_3\sqrt{\frac{n\log\frac{Nr}{n\sqrt{n}}}{N}}+\left(1+\frac{C_4\sqrt{n}}{r}\right)\frac{\sqrt{n}}{r}\varphi\left(\frac{q\sqrt{n}}{r} t\right),$$
where $q \geq 1-C_5 \sqrt{n} / {r}$.
\end{proposition}

We shall require a few Lemmas, before we proceed with the proof of Proposition \ref{key}.

\subsection{Asymptotic estimates}

For a fixed vector $\eta \in \RR^n$ and $t\in\R$, denote
\begin{equation}\label{rho}
F(\eta, t)= \frac{1}{N}\sum_{k=1}^N \varphi (r\langle \eta, \theta_k\rangle+t).
\end{equation}
Observe that $F(\eta, t) \leq 1$ with probability one. %It was proven in \cite[Lemma 3.2]{KlKol} that for any fixed $\eta\in 2B_2^n\setminus (\frac{1}{2}B_2^n)$ and $r \geq 2\sqrt{n}$, there exist constants $C,c>0$ such that
%$$ \mathbb{E} \varphi (r\langle \eta, \theta_k\rangle+t) \leq C\frac{\sqrt{n}}{r}\varphi\left(\frac{c\sqrt{n}}{r}t\right) \qquad \qquad (k=1,\ldots,N). $$
First, we shall show a sharpening of \cite[Lemma 3.2]{KlKol}.

\begin{lemma}\label{expectation}
Let $n\geq 1.$ Let $\theta$ be a random vector uniformly distributed on $\mathbb{S}^{n+2}$. For any $r>0,$ for any $t\in\R,$ for any fixed $\eta\in\R^{n+3},$ one has
$$\mathbb{E}\varphi(r\langle \theta, \eta\rangle+t)\leq \left(1+\frac{c(\log n)^2}{n}\right)\frac{\sqrt{n}}{\sqrt{n+r^2|\eta|^2}}\varphi\left(\frac{t\sqrt{n}}{\sqrt{n+r^2|\eta|^2}}\right).$$
Here $c > 0$ is an absolute constant.
\end{lemma}
\begin{proof} Observe that the formulation of the Lemma allows to assume, without loss of generality, that $|\eta|=1$: indeed, in the case $\eta=0$ the statement is straight-forward, and otherwise it follows from the case $|\eta|=1$ by scaling. The random variable $\langle \theta,\eta\rangle$ is distributed on $[-1,1]$ according to the density
$$\frac{(1-s^2)^{\frac{n}{2}}}{\int_{-1}^1 (1-s^2)^{\frac{n}{2}} ds}.$$
Recall that for any $x\in[0,1]$,
\begin{equation}\label{log}
\log(1-x)=-x-\frac{x^2}{2}+O(x^3),
\end{equation}
and hence there is an absolute constant $C>0$ such that for any $x\in [0, \frac{2\log n}{n}]$,
\begin{equation}\label{log-est}
\log(1-x)\geq -x-\frac{C(\log n)^2}{n^2}.
\end{equation}

Applying (\ref{log-est}) with $x=s^2$, we estimate
$$
\int_{-1}^1 (1-s^2)^{\frac{n}{2}} ds\geq \int_{-\sqrt{\frac{2\log n}{n}}}^{\sqrt{\frac{2\log n}{n}}}(1-s^2)^{\frac{n}{2}} ds\geq $$
$$
\int_{-\sqrt{\frac{2\log n}{n}}}^{\sqrt{\frac{2\log n}{n}}}e^{-\frac{ns^2}{2}-\frac{C(\log n)^2}{2n}} ds \geq \left(1-\frac{c'(\log n)^2}{n}\right)\int_{-\sqrt{\frac{2\log n}{n}}}^{\sqrt{\frac{2\log n}{n}}}e^{-\frac{ns^2}{2}}ds=$$
\begin{equation}\label{integral}
\frac{1}{\sqrt{n}}\left(1-\frac{c'(\log n)^2}{n}\right)\int_{-\sqrt{2\log n}}^{\sqrt{2\log n}}e^{-\frac{s^2}{2}}ds.
\end{equation}
Recall that for any $a>0$, one has
\begin{equation}\label{gausstail1}
\int_a^{\infty} e^{-\frac{y^2}{2}} dy\leq \frac{1}{a}e^{-\frac{a^2}{2}},
\end{equation}
and therefore
\begin{equation}\label{gausstail}
\int_{-\sqrt{2\log n}}^{\sqrt{2\log n}}e^{-\frac{s^2}{2}}ds\geq \sqrt{2\pi}-\frac{\sqrt{2}}{n\sqrt{\log n}}.
\end{equation}
By (\ref{gausstail}) and (\ref{integral}), we conclude that there exists an absolute constant $\tilde{c}>0$ such that
\begin{equation}\label{circ-int}
\int_{-1}^1 (1-s^2)^{\frac{n}{2}}ds\geq \frac{\sqrt{2\pi}}{\sqrt{n}}\left(1-\frac{\tilde{c}(\log n)^2}{n}\right).
\end{equation}
We remark that the second order term estimate is of course not sharp, yet it is more than sufficient for our purposes.

Next, using the inequality $1-x\leq e^{-x}$ for $x=s^2$, we estimate from above
\begin{equation}\label{integral2}
\int_{-1}^1 (1-s^2)^{\frac{n}{2}} e^{-\frac{(rs+t)^2}{2}}ds\leq \int_{-\infty}^{\infty} e^{-\frac{ns^2+(rs+t)^2}{2}}ds.
\end{equation}
It remains to observe that
$$ns^2+(rs+t)^2=\left(\sqrt{n+r^2}s+\frac{tr}{\sqrt{n+r^2}}\right)^2+\frac{nt^2}{n+r^2},$$
and to conclude, by (\ref{integral2}), that
\begin{equation}\label{integralmain}
\int_{-1}^1 (1-s^2)^{\frac{n}{2}} e^{-\frac{(rs+t)^2}{2}}ds\leq \sqrt{2\pi}\frac{1}{\sqrt{n+r^2}}\varphi\left(\frac{\sqrt{n}t}{\sqrt{n+r^2}}\right).
\end{equation}
From (\ref{circ-int}) and (\ref{integralmain}) we note, for every unit vector $\eta:$
\begin{equation}\label{expconcl}
\mathbb{E}\varphi(r\langle \theta, \eta\rangle+t)\leq \left(1+\frac{c(\log n)^2}{n}\right)\frac{\sqrt{n}}{\sqrt{n+r^2}}\varphi\left(\frac{t\sqrt{n}}{\sqrt{n+r^2}}\right).
\end{equation}
%The conclusion follows for arbitrary $\eta\neq 0$ by scaling.
\end{proof}

As an immediate corollary of Lemma \ref{expectation} and Hoeffding's inequality, we get:

\begin{lemma}\label{oneeta}
Let $N\geq n\geq 4$, $r \geq \sqrt{n}$ and $\rho\in (0,\frac{1}{3})$. There exist absolute constants $c,C, C'>0$ such that for all $\eta\in (1+\rho)B_2^n \setminus (1-2\rho)B_2^n$ and $t \in \RR, \beta > 0$,
$$P\left(F(\eta, t)>\beta+(1+c(\rho+\frac{(\log n)^2}{n}+\frac{n}{r^2}))\frac{\sqrt{n}}{r}\varphi\left(\frac{q t\sqrt{n}}{r}\right)\right)\leq e^{-C\beta^2 N},$$
where $q \geq 1-C'(\rho+\frac{n}{r^2}).$
\end{lemma}
\begin{proof} In view of Lemma \ref{hoeffding} (Hoeffding's inequality), it suffices to show that under the assumptions of the Lemma,
\begin{equation}\label{toshow}
\mathbb{E}\varphi(r\langle \theta, \eta\rangle+t)\leq (1+c(\rho+\frac{(\log n)^2}{n}+\frac{n}{r^2}))\frac{\sqrt{n}}{r}\varphi\left(\frac{q t\sqrt{n}}{r}\right).
\end{equation}
Indeed, by Lemma \ref{expectation}, for some $c_1>0,$
$$\mathbb{E}\varphi(r\langle \theta, \eta\rangle+t)\leq \left(1+\frac{c_1(\log n)^2}{n}\right)\frac{\sqrt{n}}{\sqrt{n+r^2|\eta|^2}}\varphi\left(\frac{t\sqrt{n}}{\sqrt{n+r^2|\eta|^2}}\right).$$
It remains to observe, that since $r \geq \sqrt{n}$,
$$\frac{|t|\sqrt{n}}{\sqrt{n+r^2|\eta|^2}}\geq \frac{q |t|\sqrt{n}}{r},$$
where $q=1+O(\rho+\frac{n}{r^2})$, and
$$\left(1+\frac{c_1(\log n)^2}{n}\right)\frac{\sqrt{n}}{\sqrt{n+r^2|\eta|^2}}\leq \left(1+c(\rho+\frac{(\log n)^2}{n}+\frac{n}{r^2})\right)\frac{\sqrt{n}}{r},$$
with an appropriate constant $c>0$. % Here we used that $\rho>\rho^2$.
\end{proof}

\medskip

\subsection{Union bound} Given $\rho>0$, recall  the notation $\mathcal{F}_{\rho}$ for the net from Lemma \ref{net}. Our next Lemma is a combination of the union bound with Lemma \ref{oneeta}.

\begin{lemma}[union bound]\label{estonnet}
There exist absolute constants $C_1, C_2, C'>0$ such that the following holds. Let $\rho\in (0, \frac{1}{3})$. Let $N\in [C_1n\log\frac{1}{\rho}, n^{10}]$ be an integer. Fix $r\in[C_2\sqrt{n}, n]$. Then with probability at least $1-e^{-5n}$, for every $\eta\in\mathcal{F}_{\rho}$, and for every $t\in \R$,
$$F(\eta, t)\leq C_6\sqrt{\frac{n}{N}\log\frac{1}{\rho}}+\left(1+C_7(\rho+\frac{n}{r^2}+\frac{(\log n)^2}{n}+\frac{1}{r})\right)\frac{\sqrt{n}}{r}\varphi\left(\frac{q\sqrt{n}t}{r}\right),$$
for large enough absolute constants $C_6, C_7>0,$ which depend only on $C_1$ and $C_2,$ and for $q\geq 1-C'(\rho+\frac{n}{r^2}).$
\end{lemma}
\begin{proof} Let
$$\alpha=C_6\sqrt{\frac{n}{N}\log\frac{1}{\rho}}+\left(1+C_7(\rho+\frac{n}{r^2}+\frac{(\log n)^2}{n}+\frac{1}{r})\right)\frac{\sqrt{n}}{r}\varphi\left(\frac{q\sqrt{n}t}{r}\right),$$
where $q \geq 1 - C'(\rho + n/r^2)$ and the constants shall be appropriately chosen later. Note that
\begin{equation}\label{alpha}
\alpha\geq C_6\sqrt{\frac{n}{N}\log\frac{1}{\rho}}\geq n^{-4.5}\cdot C_6\sqrt{\log 2},
\end{equation}
since $\rho\leq \frac{1}{2}$ and $N\leq n^{10}$.

Observe also that for any pair of vectors $\theta\in\sfe$, $\eta\in \mathcal{F}_{\rho}\subset 2B_2^n$ and for any $t\geq 3r$, we have
$$|r\langle \eta,\theta\rangle+t|\geq r,$$
and hence
\begin{equation}\label{trivialbound}
e^{-\frac{1}{2}(r\langle \eta,\theta\rangle+t)^2}\leq e^{-\frac{r^2}{2}}.
\end{equation}

In view of (\ref{alpha}), (\ref{trivialbound}), and the fact that $r\geq \sqrt{n}$, we have, for $t\geq 3r$:
$$F(\eta, t)\leq e^{-\frac{r^2}{2}}\leq e^{-\frac{n}{2}}\leq n^{-4.5}C_6\sqrt{\log 2} \leq\alpha,$$
where the inequality follows as long as $C_6$ is chosen to be larger than $1+o(1)$. This implies the statement of the Lemma in the range $t\geq 3r$.

Next, suppose $t\in [0,3r].$ Let $\epsilon=\frac{1}{r^2}$. Consider an $\epsilon$-net $\mathcal{N}=\{t_1,...,t_{m}\}$ on the interval $[0,3r]$
with $t_j = \epsilon \cdot j$. Note that
\begin{equation}\label{net-t}
\#\mathcal{N}\leq [3r^3]+1\leq 4r^3,
\end{equation}
since $r\geq \sqrt{n}\geq 1.$

For any $A\in \R$, for any $\epsilon>0$, and for any $t_1, t_2\in\R$ such that $|t_1-t_2| \leq \epsilon,$ we have
$$|A+t_2|^2\leq |A+t_1|^2+2\epsilon |A+t_1|+\epsilon^2,$$
and hence
\begin{equation}\label{tnet}
\varphi(A+t_1)\leq \varphi (A+t_2) e^{|A+t_1|\epsilon+\frac{\epsilon^2}{2}}.
\end{equation}

Observe that for all $t\in [0, 3r]$, for an arbitrary $\eta\in\mathcal{F}_{\rho}\subset 2B_2^n$, and any $\theta\in \sfe,$ we have
$$|r\langle \eta, \theta\rangle+t|\leq 5r,$$
and hence
\begin{equation}\label{estimate1}
e^{|r\langle \eta, \theta\rangle+t|\epsilon+\frac{\epsilon^2}{2}}\leq e^{5r\epsilon+\frac{\epsilon^2}{2}} = e^{\frac{5}{r}+\frac{1}{2r^2}}\leq 1+\frac{C'}{r},
\end{equation}
for an absolute constant $C'.$

By (\ref{tnet}) and (\ref{estimate1}), for each $t\in [0, 3r]$ there exists  $\tau\in\mathcal{N}$, such that
$$F(\eta, t)\leq (1+\frac{C'}{r}) F(\eta, \tau).$$
Therefore, by the union bound,
$$P\left(\exists t\in [0, 3r],\,\, \exists \eta\in\mathcal{F}_{\rho}:\,\, F(\eta, t)>\alpha\right)\leq$$
$$P\left(\exists \tau\in \mathcal{N},\,\, \exists \eta\in\mathcal{F}_{\rho}:\,\, F(\eta, \tau)>\frac{\alpha}{1+\frac{C'}{r}}\right)\leq$$
\begin{equation}\label{concl}
\#\mathcal{N}\cdot\#\mathcal{F}_{\rho}\cdot P\left(F(\eta, \tau)>\frac{\alpha}{1+\frac{C'}{r}}\right).
\end{equation}

By Lemma \ref{net} and (\ref{net-t}),
\begin{equation}\label{netbound}
\#\mathcal{N}\cdot\#\mathcal{F}_{\rho}\leq 4r^3\left(\frac{C}{\rho}\right)^n\leq \left(\frac{\tilde{C}}{\rho}\right)^n.
\end{equation}
We used above that $r\leq n.$

Let
$$\beta:=\left(1+\frac{C'}{r}\right)^{-1}C_6\sqrt{\frac{n}{N}\log\frac{1}{\rho}}.$$

Provided that $C_6$ and $C_7$ are chosen large enough, we have:
\begin{equation}\label{eqvation}
\frac{\alpha}{1+\frac{C'}{r}}\geq \beta+(1+c(\rho+\frac{(\log n)^2}{n}+\frac{n}{r^2}))\frac{\sqrt{n}}{r}\varphi\left(\frac{q t\sqrt{n}}{r}\right),
\end{equation}
and
\begin{equation}\label{eqvation2}
C\beta^2 N=C(1+\frac{C'}{r})^{-2}C_6^2n\log\frac{1}{\rho}\geq 5n+n\log\frac{\tilde{C}}{\rho},
\end{equation}
where $c$ and $C$ are the constants from Lemma \ref{oneeta}
and $\tilde{C}$ is the constant from (\ref{netbound}).

By Lemma \ref{oneeta}, (\ref{eqvation}) and (\ref{eqvation2}), we have
\begin{equation}\label{prob}
P\left(F(\eta, t)>\frac{\alpha}{1+\frac{C'}{r}}\right)\leq e^{-C\beta^2N}\leq e^{-5n-n\log\frac{\tilde{C}}{\rho}}.
\end{equation}

By (\ref{concl}), (\ref{netbound}) and (\ref{prob}), we conclude that the desired event holds with probability at least
$$1-\left(\frac{\tilde{C}}{\rho}\right)^ne^{-5n-n\log\frac{\tilde{C}}{\rho}}=1- e^{-5n}.$$
This finishes the proof.
\end{proof}

\medskip

\subsection{An application of random rounding and conclusion of the proof of the Proposition \ref{key}.} We begin by formulating a general fact about subgaussian random variables, which complements the estimate from Lemma \ref{expectation}.

\begin{lemma}\label{lowerbound}
Let $M\geq 10$. Let $Y$ be a subgaussian random variable with constant $\frac{1}{M}$: that is, suppose for any $s>0,$
\begin{equation}\label{condsub}
P(|Y|>s)\leq e^{-M^2s^2}.
\end{equation}
Then there exists an absolute constant $C>0$, such that for any $a\in\R,$
$$\E\varphi(Y+a)\geq \varphi(a)-\frac{C}{M}.$$
Here the expectation is taken with respect to $Y.$
\end{lemma}
\begin{proof} Since the condition (\ref{condsub}) applies for both $Y$ and $-Y$, and since $\varphi$ is an even function, we may assume, without loss of generality, that $a\geq 0$ (alternatively, we may replace $a$ with $|a|$ in the calculations below).

We begin by writing
$$
\E\varphi(Y+a)=\int_0^1 P(\varphi(Y+a)>\lambda)d\lambda=\int_0^{\infty} se^{-\frac{s^2}{2}}P(|Y+a|<s)ds\geq$$
\begin{equation}\label{exp-eq1}
\int_a^{\infty} s e^{-\frac{s^2}{2}}\left(1-P(|Y+a|\geq s)\right)ds=e^{-\frac{a^2}{2}}-\int_a^{\infty} s e^{-\frac{s^2}{2}}P(|Y+a|\geq s)ds.
\end{equation}
Note that for $s\geq a\geq 0,$ we have
\begin{equation}\label{eqprob}
P(|Y+a|\geq s)=P(Y\geq s-a)+P(-Y\geq s+a)\leq 2P(|Y|\geq s-a).
\end{equation}
By (\ref{condsub}) and (\ref{eqprob}), we estimate
$$
\int_a^{\infty} s e^{-\frac{s^2}{2}}P(|Y+a|\geq s)ds\leq 2\int_{a}^{\infty} se^{-\frac{s^2}{2}}e^{-M^2(s-a)^2}ds=$$
\begin{equation}\label{int-estimate}
2\int_0^{\infty} (t+a)e^{-\frac{(t+a)^2}{2}}e^{-M^2t^2}dt.
\end{equation}
Recall that
\begin{equation}\label{xe-x2}
(t+a)e^{-{\frac{(t+a)^2}{2}}}\leq\frac{1}{\sqrt{e}},
\end{equation}
and that
\begin{equation}\label{gaussintegral}
\int_0^{\infty} e^{-M^2t^2}dt=\frac{\sqrt{\pi}}{2M}.
\end{equation}
By (\ref{exp-eq1}), (\ref{int-estimate}), (\ref{xe-x2}) and (\ref{gaussintegral}), letting $C=\frac{\sqrt{\pi}}{\sqrt{e}}$, we have
\begin{equation}\label{finalineq}
\E\varphi(Y+a)\geq \varphi(a)-\frac{C}{M},
\end{equation}
yielding the conclusion.

%$$2\int_0^{\infty} (s+a)e^{-\frac{(2M^2+1)s^2+2as+a^2}{2}}ds=
%$$
%$$2e^{-\frac{a^2}{2}\left(1-\frac{1}{2M^2+1}\right)}\left(\int_0^{\infty}se^{-\frac{\left(\sqrt{2M^2+1}s+\frac{a}{\sqrt{2M^2+1}}\right)^2}{2}}ds+\int_0^{\infty}ae^{-\frac{\left(\sqrt{2M^2+1}s+\frac{a}{\sqrt{2M^2+1}}\right)^2}{2}}ds\right)=$$
%\begin{equation}\label{estimatewithM}
%2e^{-\frac{a^2}{2}\left(1-\frac{1}{2M^2+1}\right)}\left(\frac{1}{2M^2+1}\int_0^{\infty}te^{-\frac{\left(t+\frac{a}{\sqrt{2M^2+1}}\right)^2}{2}}dt+\frac{a}{\sqrt{2M^2+1}}\int_{\frac{a}{\sqrt{2M^2+1}}}^{\infty}e^{-\frac{t^2}{2}}dt\right).
%\end{equation}
%We estimate
%\begin{equation}\label{int-1}
%\int_0^{\infty}te^{-\frac{\left(t+\frac{a}{\sqrt{2M^2+1}}\right)^2}{2}}dt\leq \int_0^{\infty} te^{-\frac{t^2}{2}}dt=1,
%\end{equation}
%and
%\begin{equation}\label{int-2}
%\frac{a}{\sqrt{2M^2+1}}\int_{\frac{a}{\sqrt{2M^2+1}}}^{\infty}e^{-\frac{t^2}{2}}dt\leq e^{-\frac{a^2}{2(2M^2+1)}}\leq 1.
%\end{equation}

\end{proof}

Next, we shall demonstrate the following corollary of Lemma \ref{subgauss} and Lemma \ref{lowerbound}.

\begin{corollary}\label{maincor}
There exist absolute constants $C, c>0$ such that for any $M, r > 0$ and $\rho\in (0,\frac{c\sqrt{n}}{rM}]$, and for any $\xi\in\sfe,$
$$F(\xi,t)\leq \mathbb{E}_{\eta}F(\eta^{\xi},t)+\frac{C}{M},$$
with function $F$ defined in (\ref{rho}) and $\eta^{\xi}$ defined in Definition 2.2, and the expectation taken with respect to $\eta^{\xi}.$
\end{corollary}
\begin{proof} By Lemma \ref{subgauss}, for any fixed $\theta\in\sfe,$ for an absolute constant $c>0$, the random variable $r\langle \eta^{\xi}-\xi, \theta\rangle$ is subgaussian with constant $\frac{r\rho}{c\sqrt{n}}\leq \frac{1}{M}.$ Therefore, applying Lemma \ref{lowerbound} $N$ times with $Y=r\langle \eta^{\xi}-\xi, \theta_k\rangle$ and $a=r\langle \xi,\theta_k\rangle+t$, we get
	\begin{align*} \mathbb{E}_{\eta}\frac{1}{N}\sum_{k=1}^N \varphi(r\langle \eta^{\xi},\theta_k\rangle+t) & \geq \frac{1}{N}\sum_{k=1}^N \varphi(r\langle \xi,\theta_k\rangle+t)-\frac{1}{N}\sum_{k=1}^N\frac{C}{M} \\ & = \frac{1}{N}\sum_{k=1}^N \varphi(r\langle \xi,\theta_k\rangle+t)-\frac{C}{M},\end{align*}
finishing the proof.
\end{proof}

\medskip
We are ready to prove Proposition \ref{key}.

\textbf{Proof of the Proposition \ref{key}.} Let $\rho=\frac{n\sqrt{n}}{Nr}$. By Corollary \ref{maincor}, applied with $M=c\frac{N}{n}$, we have, for every $\xi\in\sfe$,
$$\frac{1}{N}\sum_{k=1}^N \varphi(r\langle \xi, \theta_k\rangle+t)\leq \E_{\eta} \frac{1}{N}\sum_{k=1}^N \varphi\left(r\langle \eta^{\xi}, \theta_k\rangle+t\right)+\frac{C'n}{N}\leq$$
\begin{equation}\label{ref}
\max_{\eta\in\mathcal{F}_{\rho}}\frac{1}{N}\sum_{k=1}^N \varphi\left(r\langle \eta, \theta_k\rangle+t\right)+\frac{C'n}{N}.
\end{equation}
By Lemma \ref{estonnet} and with our choice of $\rho$, with probability $1-e^{-5n}$, (\ref{ref}) is bounded from above by
$$C_6\sqrt{\frac{n}{N}\log\frac{Nr}{n\sqrt{n}}}+\left(1+C_7(\frac{n\sqrt{n}}{Nr}+\frac{n}{r^2}+\frac{(\log n)^2}{n}+\frac{1}{r})\right)\frac{\sqrt{n}}{r}\varphi\left(\frac{q\sqrt{n}t}{r}\right)+\frac{C'n}{N},$$
where $q=1-C'(\rho+\frac{n}{r^2})\geq 1-C_5(\frac{\sqrt{n}}{r})$, in view of our choice of $\rho$. It remains to note, in view of the fact that $N\geq nC_1 \log 2$ and $r\geq C_2\sqrt{n}$, that for an appropriate absolute constant $C_3>0$, one has
$$C_6\sqrt{\frac{n}{N}\log\frac{Nr}{n\sqrt{n}}}+\frac{C'n}{N}\leq C_3\sqrt{\frac{n}{N}\log\frac{Nr}{n\sqrt{n}}},$$
and for an appropriate absolute constant $C_4>0,$
$$C_7\left(\frac{n\sqrt{n}}{Nr}+\frac{n}{r^2}+\frac{(\log n)^2}{n}+\frac{1}{r}\right)\leq C_4\frac{\sqrt{n}}{r}.$$
The proposition follows. $\square$

\section{Proof of Theorem \ref{main}.}

%Denote $C_0=\frac{\tilde{C}}{c}$, $c_0=\frac{c}{C_0}=\frac{c^2}{\tilde{C}}$, and $K=\frac{C_2}{c_0}$, where $\tilde{C}$, $c$ and $C_2$ are the constants from Proposition \ref{key}.

In this section $C_1,\ldots,C_5 > 0$ stand for the universal constants from Proposition \ref{key}, where we recall that $C_2 \geq 2 C_5$.
Throughout this section we also set
$$ C_6 = 4 + 2 \cdot \sup_n \sum_{k=1}^{\infty} \frac{1}{\sqrt{ \log^{(k)} n }} \cdot 1_{ \{ \log^{(k)} n \geq 1 \}}. $$
Let $m$ be the largest positive integer such that
for $k=1,\ldots,m$,
$$ \left( \frac{ \log^{(k-1)} n }{\log^{(k)} n} \right)^{3/2} \geq 1 + 2 C_2 $$
and
$$ \left( \log^{(k-1)} n \right)^5 \geq C_1 \log \left[ 2 \frac{(\log^{(k-1)} n)^{13/2}}{(\log^{(k)} n)^{3/2}}  \right]. $$
Then
\begin{equation}\label{logm-good}
C_0 \leq \log^{(m)}n\leq C'_0,
\end{equation}
for some positive absolute constants $C_0, C'_0$. Let us emphasize  that also in this section, the constants $C, \tilde{C}, C', C''$ etc.  denote auxiliary positive universal constants
whose value may change from line to line.

\medskip
Consider, for $k=1,...,m$
$$N_1=n^{10},\,\,N_2=n(\log n)^5,...,\,\,N_k=n \left(\log^{(k-1)} n\right)^5,...$$
Let also
$$R_1=\frac{n}{\log n},..., \,R_k=\frac{n}{\log^{(k)}n},...$$
and
$$ \tilde{R}_k = \frac{R_k}{\sqrt{\log^{(k)} n}} = \frac{n}{(\log^{(k)} n)^{3/2}}. $$

Consider independent unit random vectors $\theta_{kj} \in {\mathbb S}^{n-1}$, where $k=1,...,m$ and $j=1,...,N_k$. Following \cite{KlKol}, consider the convex body
$$K=conv\{\pm R_k\theta_{kj},\pm ne_i\},$$
and the probability measures
$$ \mu_k = \frac{1}{N_k} \sum_{j=1}^{N_k} \delta_{\tilde{R}_k \theta_{k j}}, \qquad \qquad \mu_{-k} = \frac{1}{N_k} \sum_{j=1}^{N_k} \delta_{-\tilde{R}_k \theta_{k j}}, $$
where $\delta$ stands for the Dirac measure. We now set
$$\mu=\gamma_n * \frac{\mu_1 * \mu_2 * \ldots * \mu_m \, + \, \mu_{-1} * \mu_{-2} * \ldots * \mu_{-m} }{2}.$$
Here $\gamma_n$ stands for the standard Gaussian measure on $\R^n$. We shall show that there exists a configuration of $\theta_{kj}$, such that $\mu$ and $L=C_6 K$ satisfy the conclusion of the theorem.

\medskip
\textbf{Step 1.} Firstly, we estimate the volume of the body $L=C_6 K$ from above, following the method of \cite{KlKol}. Note that for all $k=1,...,m$ we have $\varphi\left(\frac{5n}{R_k}\right)\leq c,$ for some absolute constant $c\in (0,1)$, and hence there exists an absolute constant $\hat{C} >0$ such that
\begin{equation}\label{thelasteq}
\log \left[ 1-\varphi\left(\frac{5n}{R_k}\right) \right]\geq -\hat{C}\varphi\left(\frac{5n}{R_k}\right),
\end{equation}
for all $k=1,...,m$.

By Khatri-Sidak lemma (see, e.g. \cite{KhatSid} for a simple proof), applied together with the Blaschke-Santalo inequality, and in view of (\ref{thelasteq}), we have
$$|C_6 K|^{-1}\geq c_1^n|5nK^o|\geq c_2^n \gamma_n(5nK^o)\geq c^n\prod_{k=1}^{m} \left(1-\varphi\left(\frac{5n}{R_k}\right)\right)^{N_k}\geq$$
\begin{equation}\label{eqstep1-1}
c^n \exp\left(-\hat{C} \sum_{k=1}^{m} N_k e^{-\frac{25n^2}{2R_k^2}}\right).
\end{equation}
Plugging the values of $N_k$ and $R_k,$ and using $\frac{25}{2}>7$, we get
\begin{equation}\label{eqstep1-2}
\sum_{k=1}^{m} N_k e^{-\frac{25n^2}{2R_k^2}}\leq n^{10}e^{-7(\log n)^2}+n\sum_{k=2}^m (\log^{(k-1)}n)^5e^{-7(\log^{(k)}n)^2}\leq c'n,
\end{equation}
since the sum converges faster than exponentially.

By (\ref{eqstep1-1}) and (\ref{eqstep1-2}), we conclude that
\begin{equation}\label{1}
|C_6 K|\leq c_0^n,
\end{equation}
for some absolute constant $c_0>0.$

\medskip
\textbf{Step 2.} Next, we estimate the sections from above. Note that (see \cite{KlKol} for details),
\begin{equation} \mu^+(\xi^{\perp}+t\xi)= \frac{A + B}{2} \label{eq_1023} \end{equation}
where
\begin{equation}\label{lastlines}
A = \frac{1}{\sqrt{2\pi}}\frac{1}{N_1...N_m} \sum_{j_1=1}^{N_1} \sum_{j_2=1}^{N_2} \ldots \sum_{j_m=1}^{N_m} \varphi(t+\tilde{R}_1\langle \xi, \theta_{1j_1}\rangle+...+\tilde{R}_m\langle \xi,\theta_{mj_m}\rangle)
\end{equation}
and
\begin{equation}\label{lastlines2}
B = \frac{1}{\sqrt{2\pi}}\frac{1}{N_1...N_m} \sum_{j_1=1}^{N_1} \sum_{j_2=1}^{N_2} \ldots \sum_{j_m=1}^{N_m} \varphi(-t+\tilde{R}_1\langle \xi, \theta_{1j_1}\rangle+...+\tilde{R}_m\langle \xi,\theta_{mj_m}\rangle)
\end{equation}
We are going to apply Proposition \ref{key} repeatedly in order to bound $A$ and $B$. For this purpose we introduce some notation.
For $r \geq  C_2 \sqrt{n}$ we set
\begin{equation}  q(r) = 1 - C_5 \frac{\sqrt{n}}{r}  \in [1/2,1], \label{eq_1735} \end{equation}
 where $C_5$ is the constant coming from Proposition \ref{key}. We define $r_1,r_2,\ldots,r_m \in [C_2 \sqrt{n}, n]$ such that $$ r_1 := \tilde{R}_1 $$ and
for $k =1,\ldots,m$,
\begin{equation}  r_{k+1} := \frac{q(r_k) \sqrt{n} \tilde{R}_{k+1}}{\tilde{R}_k}. \label{eq_1143} \end{equation}
From (\ref{eq_1143}),
\begin{equation}  r_{k+1} = \frac{\tilde{R}_1}{\tilde{R}_{k+1}} \cdot \frac{\tilde{R}_{k+1} r_{k+1}}{r_1} =
\left( \prod_{j=1}^k \frac{q(r_j) \sqrt{n}}{r_{j+1}} \right)\cdot \frac{\tilde{R}_{k+1} r_{k+1}}{r_1} = \left( \prod_{j=1}^k \frac{q(r_j) \sqrt{n}}{r_j} \right)\cdot \tilde{R}_{k+1}. \label{eq_1144} \end{equation}
Denote
\begin{equation}\label{CK}
\alpha_k:= \prod_{j=1}^{k-1} \left[ \left(1+\frac{C_4\sqrt{n}}{r_j} \right)
\frac{1}{q(r_j)} \right] \leq \prod_{j=1}^{k-1} \left(1+\frac{\hat{C}\sqrt{n}}{r_j} \right).
\end{equation}
By (\ref{eq_1144}), (\ref{eq_1143}) and (\ref{eq_1735}) the quantity $\alpha_k$ satisfies
\begin{equation}  \prod_{j=1}^{k-1} \left[ \left(1+\frac{C_4\sqrt{n}}{r_j} \right) \frac{\sqrt{n}}{r_j} \right]
= \frac{r_k \alpha_k}{\tilde{R}_k}  \leq 2 \frac{\alpha_k \sqrt{n}}{\tilde{R}_{k-1}}. \label{eq_1157} \end{equation}
From  (\ref{eq_1143}) we see that for $j \geq 2$ we have $\sqrt{n} / r_j \leq 2 \left( \log^{(j)} n / \log^{(j-1)} n \right)^{3/2}$.
Hence by (\ref{CK}),  for every $k=1,...,m$,
\begin{equation}\label{CK-bound}
\alpha_k\leq \left(1+ \hat{C} \frac{\sqrt{n}}{\tilde{R}_1} \right)\prod_{j=2}^{k-1} \left(1+\check{C} \left( \frac{\log^{(j)}n}{ \log^{(j-1)}n} \right)^{3/2} \right)\leq C e^{\bar{C} \sum_{j=2}^{k-1} \left( \frac{\log^{(j)}n}{\log^{(j-1)}n} \right)^{3/2} }\leq \tilde{C},
\end{equation}
since the sum converges faster than exponentially.

\medskip By the definition of $m$ we have that for each $k=1,\ldots,m$, the pair $N=N_k$ and $r=r_k$ satisfies the assumptions of Proposition \ref{key}. Applying Proposition \ref{key} consecutively $m$ times with $N=N_k$ and $r=r_k$ for $k=1,...,m$, we get that with probability at least $1-m e^{-5n}=1-o(1),$ for every $\xi\in\sfe$ and for every $t\in \R,$ the term $A$ from (\ref{lastlines}) is bounded from above by a constant multiple of
%$$\frac{n}{N_1}+c\frac{\sqrt{n}}{R_1}\left(\frac{n}{N_2}+\frac{\sqrt{n}}{R_2}\left($$
$$\sqrt{\frac{n\log\frac{N_1 r_1}{n\sqrt{n}}}{N_1}}+ \frac{\alpha_2 \sqrt{n}}{\tilde{R}_1} \sqrt{\frac{n\log\frac{N_2r_2}{n\sqrt{n}}}{N_2}}
+\ldots+ \frac{\alpha_m \sqrt{n}}{\tilde{R}_{m-1}} \sqrt{\frac{n\log\frac{N_mr_m}{n\sqrt{n}}}{N_m}} + \frac{\alpha_{m+1} \sqrt{n}}{\tilde{R}_m}
 $$
$$ \leq \frac{C'}{\sqrt{n}}+\frac{C''}{\sqrt{n}} \left( \sum_{k=2}^{m} \alpha_k\frac{\sqrt{\log^{(k)} n}}{\log^{(k-1)} n} \right) +\frac{\alpha_m\log^{(m)}n}{\sqrt{n}}\leq \frac{C}{\sqrt{n}}. $$
Here we used (\ref{CK-bound}) to bound $\alpha_k$, and (\ref{logm-good}) to bound $\log^{(m)}n$,
as well as the estimates
$$ \frac{N_k r_k}{n \sqrt{n}} \leq C \frac{(\log^{(k-1)} n)^{13/2}}{(\log^{(k)} n)^{3/2}} \leq \tilde{C} (\log^{(k-1)} n)^{13/2} $$
and that for $k \geq 2$,
$$ \frac{n}{\tilde{R}_{k-1} \sqrt{N_k}} = \frac{1}{\sqrt{n} \cdot \log^{(k-1)} n}. $$

The same bound applies also to the term $B$ from (\ref{lastlines2}).
We conclude, in view of (\ref{eq_1023}) that with high probability, for all $\xi\in\sfe$ and for all $t\in\R$,
\begin{equation}\label{2}
\mu(\xi^{\perp}+t\xi)\leq \frac{C}{\sqrt{n}}.
\end{equation}

\textbf{Step 3.} Recall that $\mu$ is an average of translates of the standard Gaussian measure that are centered at points of the form
$$ \pm \sum_{k=1}^m \tilde{R}_k \theta_{kj_k} = \pm \sum_{k=1}^m \frac{1}{\sqrt{ \log^{(k)} n} }  R_k \theta_{kj_k}
\in \left( \sum_{k=1}^m \frac{1}{\sqrt{ \log^{(k)} n} } \right) \cdot K \subseteq \frac{C_6}{2} K.  $$
Similar to as it was shown in \cite[Lemma 3.8]{KlKol}, using the facts that $\sqrt{n}B_2^n\subset K$ and $C_6 \geq 4$ and therefore $(C_6/2) K \supseteq 2\sqrt{n}B_2^n$,
one has
\begin{equation}\label{3}
\mu(C_6 K)\geq \mu \left( \frac{C_6}{2} K + 2 \sqrt{n} B_2^n \right) \geq  \gamma_n(2\sqrt{n}B_2^n) \geq\frac{1}{2},
\end{equation}
where, e.g. Markov's inequality is used in the last passage.
Combining (\ref{1}), (\ref{2}) and (\ref{3}), we arrive to the conclusion of the theorem, with $L=C_6 K$. $\square$

\section{Further applications}

\subsection{Comparison via the Hilbert-Schmidt norm for arbitrary matrices}

As another consequence of the Lemma \ref{subgauss}, we have:

\begin{lemma}[comparison via the Hilbert-Schmidt norm]\label{HS}
Let $\rho\in (0,\frac{1}{2}).$ There exists a collection of points $\mathcal{N}\subset 2B_2^n\setminus \frac{1}{2}B_2^n$ with $\#\mathcal{N}\leq (\frac{C}{\rho})^n$ such that for any matrix $A:\R^n\rightarrow \R^N$, for every $\xi\in\sfe$ there exists an $\eta\in\mathcal{N}$ satisfying
\begin{equation}\label{HS-comp}
|A\eta|^2\leq C_1|A\xi|^2+C_2\frac{\rho^2}{n}||A||_{HS}^2.
\end{equation}
Here $C, C_1, C_2$ are absolute constants.

\end{lemma}
\begin{proof} Recall that $|Ax|^2=\sum_{i=1}^N \langle X_i, x\rangle^2$, where $X_i$ are the rows of $A$. In order to prove the Lemma, it suffices to show, for every vector $g\in\R^n$, that
\begin{equation}\label{EG}
\mathbb{E}_{\eta} \langle \eta^{\xi}, g\rangle^2\leq C_1\langle \xi, g\rangle^2+C_2\frac{\rho^2}{n}|g|^2;
\end{equation}
the Lemma shall follow by applying (\ref{EG}) to the rows of $A$ and summing up.

We shall show (\ref{EG}). Using the inequality $a^2=(a-b+b)^2\leq 2(a-b)^2+2b^2,$ we see
$$|\langle \eta^{\xi}, g\rangle|^2\leq 2|\langle \eta^{\xi}, g\rangle-\langle \xi, g\rangle|^2+2|\langle \xi, g\rangle|^2,$$
and hence
\begin{equation}\label{expectat}
\E_{\eta}|\langle \eta^{\xi}, g\rangle|^2\leq 2\E_{\eta}|\langle \eta^{\xi}, g\rangle-\langle \xi, g\rangle|^2+2|\langle \xi, g\rangle|^2.
\end{equation}

By Lemma \ref{subgauss}, $|\langle \eta^{\xi}, g\rangle-\langle \xi, g\rangle|$ is sub-gaussian with constant $c'\frac{\rho|g|}{\sqrt{n}}$, and hence
\begin{equation}\label{exp}
\mathbb{E}_{\eta}|\langle \eta^{\xi}, g\rangle-\langle \xi, g\rangle|^2\leq 2\int_0^{\infty} t e^{-\frac{cnt^2}{\rho^2|g|^2}}dt\leq C\frac{\rho^2|g|^2}{n},
\end{equation}
for some absolute constant $C>0$; (\ref{expectat}) and (\ref{exp}) entail (\ref{EG}).

% for any $t$ such that $t\geq |\langle \xi, g\rangle|,$
%$$P(|\langle \eta^{\xi}, g\rangle|\geq t)\leq e^{-\frac{cn(t-|\langle \xi, g\rangle|)^2}{\rho^2 |g|^2}}.$$
%Therefore, using the by parts representation of the second moment and the change of variables, we get
%$$\mathbb{E}_{\eta}\langle \eta^{\xi}, g\rangle^2=2\int_0^{\infty} t P(|\langle \eta^{\xi}, g\rangle|\geq t)dt=$$$$2\int_0^{|\langle \xi, g\rangle|} t P(|\langle \eta^{\xi}, g\rangle|\geq t)dt+2\int_{|\langle \xi, g\rangle|}^{\infty} t P(|\langle \eta^{\xi}, g\rangle|\geq t)dt \leq$$ $$\langle \xi, g\rangle^2+2\int_0^{\infty} t e^{-\frac{cn(t-|\langle \xi, g\rangle|)^2}{\rho^2 |g|^2}}dt=c\frac{\rho |g|}{\sqrt{n}}|\langle \xi, g\rangle| \int_0^{\infty} e^{-s^2} ds+c'\frac{\rho^2 |g|^2}{n}\int_0^{\infty} se^{-s^2}ds=$$
%$$c_1\frac{\rho |g|}{\sqrt{n}}|\langle \xi, g\rangle|+c_2\frac{\rho^2 |g|^2}{n}\leq C_1\langle \xi, g\rangle^2+C_2\frac{\rho^2}{n}|g|^2,$$
%where in the last passage we used Cauchy's inequality. That finishes the proof.
\end{proof}

A fact similar to Lemma \ref{HS} was recently shown and used by Lytova and Tikhomirov \cite{LytTikh}.

Lemma \ref{HS} shows that there exists a net of cardinality $C^n$, such that for any random matrix $A:\R^n\rightarrow\R^N$ whose entries have bounded second moments, with probability at least
$$1-P(||A||_{HS}^2\geq 10\mathbb{E}||A||_{HS}^2)\geq \frac{9}{10}$$
one has (\ref{HS-comp}), with $\mathbb{E}||A||_{HS}^2$ in place of $||A||_{HS}^2$.  However, such probability estimate is unsatisfactory when studying small ball estimates for the smallest singular values of random matrices. In the soon-to-follow paper, we significantly strengthen Lemma \ref{HS}: we employ the idea of Rebrova and Tikhomirov \cite{RebTikh}, and in place of the covering by cubes, we consider a covering by paralelepipeds of sufficiently large volume. This leads us to consider the following \emph{refinement of the Hilbert-Schmidt norm}: with $\kappa>1$, for an $N\times n$ matrix $A$, define
$$\B_{\kappa}(A)=\min_{\alpha_i\in[0,1],\,\prod_{i=1}^n\alpha_i\geq \kappa^{-n}} \sum_{i=1}^n \alpha_i^2 |Ae_i|^2.$$
$\B_{\kappa}$ acts as an averaging on the columns of $A$. In a separate paper we shall show that there exists a net $\mathcal{N}\subset 2B_2^n\setminus \frac{1}{2}B_2^n$, of cardinality $\left(\frac{C}{\rho}\right)^n$, such that for all $N\times n$ matrices $A$, for every $\xi\in\sfe$ there exists an $\eta\in\mathcal{N}$ satisfying
\begin{equation}\label{HS-comp-imp}
|A\eta|^2\leq C_1|A\xi|^2+\frac{\rho^2}{n}\B_{10}(A).
\end{equation}
The proof shall be a combination of the argument similar to the proof of Lemma \ref{HS} along with the construction of a net on the family of admissible nets. The bound on the cardinality of that net shall follow, in fact, again from Lemma \ref{net}.

The advantage of (\ref{HS-comp-imp}) over (\ref{HS-comp}) is the strong large deviation properties of $\B_{10}(A)$. For example, we shall show an elementary fact that for any random matrix $A$ with independent columns and $\mathbb{E}||A||_{HS}^2<\infty$,
\begin{equation}\label{B-deviation}
P(\B_{10}(A)\geq 2\mathbb{E}||A||_{HS}^2)\leq e^{-cn}.%\left(\frac{(\mathbb{E}||A||_{HS}^2)^2}{n\sum (\mathbb{E}|Ae_i|^2)^2}\right)},
\end{equation}
%which in the case of matrices with i.i.d. entries becomes $e^{-cn}$.
The detailed proofs of the mentioned facts, and applications to sharp estimates for the small ball probability of the smallest singular value of heavy-tailed matrices shall be outlined in a separate paper.

\subsection{Covering spheres with strips}

For $\theta\in\sfe$, $\tau\in\R$ and $\alpha>0,$ consider a strip
$$S(\theta,\alpha, \tau):=\{\xi\in\sfe: |\langle \xi, \theta\rangle +\tau|\leq \alpha\}.$$
Observe that
$$\sum_{k=1}^N 1_{S(\theta_k, \frac{1}{r}, \frac{t}{r})}(\xi)\leq  C \sum_{k=1}^N \varphi (r\langle \xi, \theta_k\rangle +t).$$

Therefore, Proposition \ref{key} implies

\begin{proposition}\label{naszodi}
For any $N$ and for any $\alpha\leq \frac{c}{\sqrt{n}}$
with $N \in [c n \log \frac{N}{\alpha n^{3/2}}, n^{10}]$
there exists a collection of points $\theta_1,...,\theta_N\in\sfe$ such that every strip of width $2\alpha$ contains no more than
$$
\tilde{C}\left[\sqrt{Nn\log\frac{N}{\alpha n^{3/2}}}+N\sqrt{n}\alpha\right]
$$
points in this collection.
\end{proposition}

We note that in view of the point-strip duality, bounding $\sum_{k=1}^N 1_{S(\theta_k, \frac{1}{r}, \frac{t}{r})}(\xi)$ yields estimates of the form stated in Proposition \ref{naszodi}.

The direct consideration of the characteristic functions in place of the Gaussian functions gives exactly the same bound as an application of Proposition \ref{key}.

In \cite{FNN}, Frankl, Nagy and Naszodi conjecture that for every collection of $N$ points on $\mathbb{S}^2$ there exists a strip of width $\frac{2}{N}$ containing at least $f(N)$ points, where $f(N)\rightarrow\infty$ as $N\rightarrow \infty$. %By considering projections and injections, it is clear that disproving this conjecture in arbitrary dimension would entail the two-dimensional version. Conversely, proving this conjecture in $\R^2$ suffices to show the same statement in any dimension. %It seems that the question is very difficult, and dimension plays very little role.
Proposition \ref{naszodi} generalizes Theorem 4.2 by Frankl, Nagy, Naszodi \cite{FNN} from the two-dimensional case to an arbitrary dimension, with good dimensional constant, although it does not shed any light on the dependence on $N$.

\end{document}